\documentclass[12pt]{article}
\usepackage{amssymb,amsfonts,amsthm,amsmath,mathrsfs,xspace,hyperref}
\usepackage{graphicx}
\usepackage[francais,english]{babel}
\usepackage{inputenc}
\usepackage[T1]{fontenc}
\usepackage{authblk}
\usepackage[margin=1.45in, centering]{geometry}
\usepackage{csquotes}
\usepackage{color}
\usepackage{enumitem}

\usepackage[textsize=footnotesize,textwidth=15ex]{todonotes}

\usepackage{titlesec}
\titleformat{\subsection}[runin]{\normalfont\bfseries}{\thesubsection}{0.5em}{}[.]
\titleformat{\subsubsection}[runin]{\normalfont\bfseries}{\thesubsubsection}{0.5em}{}[.]

	\newcommand{\nn}{\mathbf{n}}
\newcommand{\kk}{\mathbf{k}}

\newcommand{\Sub}{\ensuremath{\mathbf{Sub}}}%

        \renewcommand{\H}{\ensuremath{\mathcal{H}}}%
                \newcommand{\K}{\ensuremath{\mathcal{K}}}%

  \newcommand{\Sym}{\ensuremath{\operatorname{Sym}}}%
      \newcommand{\Alt}{\ensuremath{\operatorname{Alt}}}%
	\newcommand{\Fix}{\ensuremath{\textrm{Fix}}}%
  \newcommand{\sub}{\ensuremath{\mathrm{Sub}}}%
  \newcommand{\stab}{\ensuremath{\operatorname{Stab}}}%
        \newcommand{\Aut}{\ensuremath{\operatorname{Aut}}}%
    \newcommand{\rist}{\ensuremath{\operatorname{Rist}}}%
		     \newcommand{\Stab}{\ensuremath{\operatorname{Stab}}}%

\newcommand{\Comm}{\mathrm{Comm}}

\theoremstyle{definition}
  \newtheorem{defin}{Definition}[section]

  \newtheorem{problem}[defin]{Problem}

\theoremstyle{plain}
  \newtheorem{thm}[defin]{Theorem}
  \newtheorem{main thm}{Theorem}
  \newtheorem{prop}[defin]{Proposition}
    \newtheorem{prop-def}[defin]{Proposition-Definition}

  \newtheorem{cor}[defin]{Corollary}
\newtheorem{lem}[defin]{Lemma}

\theoremstyle{remark}
  \newtheorem{rmk}[defin]{Remark}

\title{Piecewise strongly proximal actions, \\
free  boundaries and the Neretin groups}

\author{Pierre-Emmanuel Caprace\thanks{F.R.S.-FNRS Senior Research Associate. \texttt{pe.caprace@uclouvain.be}}}

\author{Adrien Le Boudec\thanks{CNRS Researcher. \texttt{adrien.le-boudec@ens-lyon.fr}}} 

\author{Nicol\'as Matte Bon\thanks{CNRS Researcher. \texttt{mattebon@math.univ-lyon1.fr}}} 
\affil{UCLouvain, 1348 Louvain-la-Neuve, Belgium}
\affil{UMPA - ENS Lyon, France}
\affil{ICJ - Universite de Lyon I, France}

\date{December 8, 2022}

\begin{document}
	
\newgeometry{margin=1.1in}
	
	\maketitle
	

%

\begin{abstract}
A closed subgroup $H$ of a locally compact group $G$ is confined if the closure of the conjugacy class of $H$ in the Chabauty space of $G$ does not contain the trivial subgroup. We establish a dynamical criterion on the action of a totally disconnected locally compact group $G$ on a compact space $X$ ensuring that no relatively amenable subgroup of $G$ can be confined. This property is equivalent to the fact that the action of $G$ on its Furstenberg boundary is  free. Our criterion applies to the Neretin groups. We  deduce that each Neretin group has two inequivalent irreducible unitary representations that are weakly equivalent. This implies that the Neretin groups are  not of type I, thereby answering  a question of Y.~Neretin.
\end{abstract}


%

\section{Introduction}

Let $G$ be a locally compact group. A compact $G$-space $X$ is a compact space equipped with a continuous action of $G$. The action of $G$ on $X$ is \textbf{strongly proximal} if for every $\mu \in \mathrm{Prob}(X)$, the closure of the $G$-orbit of $\mu$ in $\mathrm{Prob}(X)$ contains a Dirac measure, where the space $\mathrm{Prob}(X)$ of Borel probability measures on $X$ is endowed with the weak$^\ast$ topology. The $G$-space $X$ is a (topological) \textbf{$G$-boundary} if the $G$-action is minimal and strongly proximal \cite{Furst-bnd-theory}. If $G$ is an amenable group, the only $G$-boundary is the one-point space; and this property actually characterizes amenability among locally compact groups. 

Every group $G$ admits a  $G$-boundary, unique up to isomorphism, with the property that every $G$-boundary is a factor of it. It is called the \textbf{Furstenberg boundary} of $G$; we denote it by  $\partial_F G$. The group $G$ acts faithfully on $\partial_F G$ if and only if the only amenable normal subgroup of $G$ is the trivial subgroup $\langle e \rangle$ \cite{Furman-min-strg}. In this note we are interested in the following:

\begin{problem} \label{prob-intro}
Determine when the action of $G$ on $\partial_F G$ is free.
\end{problem}

A key motivation for that question is that the freeness of the $G$-action on $\partial_F G$  is equivalent to various other properties of the group $G$. We say that the action of $G$ on a minimal compact $G$-space $X$ is \textbf{topologically free} if there is a dense set of points in $X$ that have a trivial stabilizer in $G$.  We say that   $H$ is \textbf{relatively amenable} in $G$ if $H$ fixes a probability measure on  every compact $G$-space. Clearly, every amenable subgroup is relatively amenable; when $G$ is discrete, the converse holds, see \cite{Cap-Mon-RelMoy}.  A \textbf{uniformly recurrent subgroup} (or \textbf{URS} for short) of $G$ is a minimal $G$-invariant closed subset of the Chabauty space $\sub(G)$ of closed subgroups of $G$ (we recall that $\sub(G)$  is compact). A closed subgroup $H \leq G$ is  \textbf{confined} in $G$ if the closure of the conjugacy class of $H$ in $\sub(G)$ avoids the trivial subgroup $\langle e\rangle$.

\begin{thm}[{See \cite{KK, BKKO, Kennedy} in the case of discrete groups}]\label{thm:equiv-free}
Let $G$ be a locally compact group. The following conditions are equivalent. 
\begin{enumerate}[label=(\roman*)]
\item The $G$-action on $\partial_F G$ is free.
\item There is a $G$-boundary on which the $G$-action is topologically free. 
\item No relatively amenable closed subgroup of $G$ is confined.
\item The  only relatively amenable URS of $G$ is the trivial subgroup. 
\end{enumerate}
If in addition $G$ is discrete, then those are also equivalent to:
\begin{enumerate}[resume*]
\item $G$ is  \textbf{$C^\ast$-simple}, i.e.\ the reduced $C^\ast$-algebra of $G$ is a simple $C^\ast$-algebra. 
\end{enumerate}
\end{thm}

In the case where the group $G$ is discrete, it follows  from the recent works \cite{KK, BKKO, Kennedy} that the five conditions in Theorem~\ref{thm:equiv-free} are equivalent. For a general  (e.g. indiscrete) locally compact group $G$, the implications  $\textrm{(i)} \Rightarrow \textrm{(ii)} \Leftrightarrow \textrm{(iii)} \Leftrightarrow \textrm{(iv)}$ are rather straightforward, at least if $G$ is second countable (see Proposition~\ref{prop:rel-amen-confined} and Remark~\ref{rem:2nd-countable} below), while the converse implication $\textrm{(i)} \Leftarrow \textrm{(ii)}$ is a particular case of a result from \cite{LB-Tsankov}.  However, it is an important open problem to determine whether (v) is also equivalent  to (i)--(iv). It is worth noting that each of those conditions implies that the group $G$ is totally disconnected (see \cite{Raum} for (v), and the discussion in the following paragraphs for (iv)).

Conditions (iii) and (iv) highlight an essential feature of Problem \ref{prob-intro}, namely that it can be reformulated in terms of the $G$-action on the space of its closed subgroups. Indeed, this allows one to avoid constructing explicitly any $G$-boundary, and simply study the conjugation action of $G$ on its relatively amenable subgroups. For example, if $G$ is a discrete hyperbolic group, every amenable subgroup is virtually cyclic, hence finitely generated. In particular $G$ has countably many amenable subgroups, and it is a general fact that for a group $G$ with this property, the only amenable URS of $G$ is trivial, provided $G$ has no  amenable normal subgroup other than the trivial subgroup. This situation also covers CAT(0) groups by \cite[Cor.~B]{AdamsBallmann}. Many other discrete groups admitting an isometric action satisfying a combination of weak forms of properness and of non-positive curvature can be proved to have a free Furstenberg boundary in a similar (but more elaborate) way, see \cite{dlHarpe, DGO, BKKO}. Note that the requirement of a certain form of properness cannot be dropped according to \cite{LB-cs}. The reformulation of Problem \ref{prob-intro} in terms of confined subgroups has also been exploited in the realm of discrete groups of dynamical origin in \cite{LBMB-sub-dyn}.

For non-discrete groups, it turns out that Problem~\ref{prob-intro} has a very different flavour. Indeed, contrary to the discrete case where numerous familiar groups have no non-trivial amenable URS, many natural non-discrete locally compact groups do admit a non-trivial relatively amenable URS. This is notably the case for semisimple Lie groups and semisimple algebraic groups over local fields: any such group $G$ indeed has a cocompact amenable subgroup $P$. By cocompactness, the conjugacy class of $P$ is closed, so $P$ must be confined. Beyond the classical case, in a locally compact group $G$ acting properly and strongly transitively on a locally finite building of arbitrary (not necessarily Euclidean) type, every maximal compact subgroup is confined (this follows from \cite[Th.~4.10]{CapLec}). Thus, many natural examples of non-discrete locally compact groups fail to satisfy the condition of Problem~\ref{prob-intro}. Note that the above fact for semisimple Lie groups together with  \cite[Th. 3.3.3]{Bur-Mon-bound-coho-rigid} imply that every locally compact group $G$ such that the  only relatively amenable URS of $G$ is the trivial subgroup, must be totally disconnected.

The first goal of this note is to   contribute to Problem~\ref{prob-intro} by establishing a sufficient criterion for a locally compact group $G$ that ensures a positive answer to Problem \ref{prob-intro}. In view of the preceding discussion, there is no loss of generality in restricting to the case where $G$ is totally disconnected. Given a compact $G$-space $X$ and a clopen subset $\alpha$ of $X$, the rigid stabilizer of $\alpha$ in $G$ is the pointwise fixator of $X \! \setminus \! \alpha$. It is denoted by $R_G(\alpha) = \Fix_G(X \! \setminus \! \alpha)$.  We say that the action of $G$ on $X$ is \textbf{piecewise minimal-strongly-proximal} if the action of $R_G(\alpha)$ on $\alpha$ is minimal and strongly proximal for every non-empty clopen subset $\alpha$ of $X$. 

\begin{thm}\label{thm:intro}
	Let $G$ be a totally disconnected locally compact group. Suppose that there exists a totally disconnected compact $G$-space $X$ such that the $G$-action on $X$ is faithful and 
	piecewise minimal-strongly-proximal. Then $G$ does not have any relatively amenable confined subgroup. Equivalently, $G$ acts freely on its Furstenberg boundary $\partial_F G$.
\end{thm}

Note that the piecewise minimal-strongly-proximal property of the $G$-action  on $X$ implies in particular that $X$ is a $G$-boundary, and also that this action is very far from being free. Hence the meaning of Theorem~\ref{thm:intro} is that the existence of a $G$-boundary that is non-free and satisfies a certain strong compressibility property, ensures the existence of another $G$-boundary where the $G$-action is free.

Although the reformulation of Problem~\ref{prob-intro} in terms of confined subgroups is helpful, we emphasize that the Chabauty space $\sub(G)$ and its  $G$-invariant closed subspaces are typically delicate to describe. Moreover, the general properties of the space $\sub(G)$ are often more subtle in the case of non-discrete groups (see for instance \cite[\S 20.1]{Cap-Mon-problemlist}; see also \cite[\S1.2]{FracGela} for a remarkable recent result ensuring that, in a simple Lie group of rank~$\geq 2$, every discrete confined subgroup is a lattice). 
 In the case of discrete groups, Theorem~\ref{thm:intro} is already known \cite[Cor. 3.6]{LBMB-sub-dyn}. It is actually consequence of the following more general result: if $G$ is a discrete group and $X$ a faithful $G$-space, then for every confined subgroup $H$ of $G$, there exists a non-empty open subset $\alpha$ of $X$ such that $H$ contains the commutator subgroup of $R_G(\alpha)$ \cite[Th. 1.1]{LBMB-comm-lemma}. Here the assumption made in Theorem \ref{thm:intro} implies that $R_G(\alpha)$ is non-amenable, so in this situation it follows in particular that $H$ is not amenable either.  However it is worth noting that, as shown by classical examples, the stronger conclusion of \cite[Th. 1.1]{LBMB-comm-lemma} completely fails for non-discrete groups; see Section \ref{sec-loc-bndry}.

 Examples of groups to which Theorem \ref{thm:intro} applies are the \textbf{Neretin groups} $\mathcal N_{d,k}$ of almost automorphisms of quasi-regular tree $\mathcal T_{d, k}$. The groups $\mathcal N_{d,k}$ are non-discrete, compactly generated, simple, totally disconnected locally compact groups  \cite[\S6.3]{CapDM}. We refer to \cite{GarnLaza}  for details. The groups $\mathcal N_{d,k}$ can be defined as groups of homeomorphisms of the space of ends $\partial \mathcal T_{d, k}$, and the action of $\mathcal N_{d,k}$ on $\partial \mathcal T_{d, k}$ is piecewise minimal-strongly-proximal. The following result is a direct consequence of Theorem~\ref{thm:intro}. 
 
 \begin{cor}\label{cor:intro:Neretin-confined}
 For all integers $d, k \geq 2$, the Neretin group $\mathcal N_{d,k}$ does not have any confined relatively amenable subgroup. In particular $\mathcal N_{d,k}$ does not have any cocompact amenable subgroup. 
 \end{cor}
 
 To the best of our knowledge, the Neretin group  $\mathcal N_{d,k}$ is the first known example of a non-discrete compactly generated simple locally compact group acting freely on its Furstenberg boundary. It is likely that Theorem~\ref{thm:intro} will apply to many other  simple groups. 
 
Using that result, we  establish the following representation theoretic property of the Neretin group.

\begin{thm}\label{thm:nonType1}
For all integers $d, k \geq 2$, the Neretin group $\mathcal N_{d, k}$ is not a type~I group.
\end{thm}

This answers negatively Question~1.4(2) from \cite{Neretin}. We recall that a locally compact group $G$ is of \textbf{type I} is for every unitary representation $\pi$, the von Neumann algebra $\pi(G)''$ is of type I. By Glimm's theorem \cite{Glimm1961}, a second countable group $G$ is of type I if and only if any two weakly equivalent irreducible unitary representations of $G$ are unitarily equivalent. We refer to \cite{Dixmier} and \cite{BekHar} for  detailed expositions.

Let us also mention that Y.~Neretin  has proved in \cite[Th.~1.2]{Neretin} that the group $ \mathcal N_{d, d+1}$ has an open subgroup $A$ such that $(\mathcal N_{d, d+1}, A)$ forms a generalized Gelfand pair.
 Since $\mathcal N_{d, d+1}$ admits no cocompact amenable subgroup by Corollary~\ref{cor:intro:Neretin-confined}, we deduce that N.~Monod's result  on Gelfand pairs  \cite{Monod} cannot be extended to generalized Gelfand pairs, even among simple groups (see Remark~\ref{rem:Gelfand}). 
 
Motivated by the understanding of the confined subgroups of $\mathcal N_{d, k}$, we also provide a complete classification of the closed cocompact subgroups of $\mathcal N_{d, k}$, inspired by \cite{BCGM}. This classification says that there are as few proper cocompact subgroups in $\mathcal N_{d, k}$ as one might hope: any such subgroup is a finite index open subgroup of the stabilizer of an end $\xi \in \partial \mathcal T_{d, k}$, see Theorem~\ref{thm:cocoNer}. This description notably implies that $\mathcal N_{d, k}$ is an isolated point of its Chabauty space (see Corollary~\ref{cor:Neretin=Chab-isol}). This last phenomenon contrasts with the case of the automorphism group of a regular tree; see Remark \ref{rmk-aut-non-isolated}. 
 
 \bigskip
 
  The proof of Theorem \ref{thm:intro} is presented in Section \ref{sec-loc-bndry}. It is fairly elementary. The argument actually establishes non-confinement for an appropriate class of subgroups of $G$, which properly contains the class of relatively amenable subgroups. It uses in an essential way that the fixator $G_F$ of a finite subset $F$ of $X$ in $G$ admits a natural subgroup $J$ that is commensurated in $G_F$, and which is an exhaustion of subgroups that are built up from a compact open subgroup of $G_F$ and a certain rigid stabilizer in $G$. The proof consists in combining general considerations on confined subgroups (Lemma \ref{lem-conf-fixedpts} and Proposition \ref{prop-H-not-intersect-open}) together with an approximation argument at the level of the above subgroup $J$. The application to the Neretin groups and the description of the closed cocompact subgroups of the Neretin groups are given in Section \ref{sec-neretin}.
 
\subsection*{Acknowledgement}

We thank Sven Raum for his comments on an earlier version of this paper. This work was supported by ANR-19-CE40-0008-01 AODynG.

\section{Preliminaries}

Let $G$ be a locally compact group. The Chabauty space of closed subgroups of $G$, which is compact, is denoted by $\sub(G)$. We refer to \cite[Ch.~VIII, \S5]{Bourbaki_Int7-8} for a description of some of its basic properties. 

Recall that if $X$ is a compact $G$-space, the stabilizer map $X \to \sub(G)$, $x \mapsto G_x$, is upper semi-continuous, meaning that for every net $(x_i)$ in $X$ converging to $x$ and such that $(G_{x_i})$ converges to $L$ in $\sub(G)$, one has $L \leq G_x$.

Recall that a closed subgroup $H$ of a locally compact group $G$ is \textbf{relatively amenable} in $G$ if $H$ fixes a probability measure on every compact $G$-space \cite{Cap-Mon-RelMoy}. The following is well-known (it is implicit in \cite{Cap-Mon-RelMoy}, and also appears in \cite{Monod}). 

\begin{prop}
A subgroup $H$ of $G$ is relatively amenable in $G$ if and only if $H$ fixes a probability measure on $\partial_F G$.
\end{prop}

A closed subgroup $H \leq G$ is called \textbf{confined} if the closure of its conjugacy class in $\sub(G)$ avoids the trivial subgroup $\{e\}$.

\begin{prop}\label{prop:rel-amen-confined}
Let $G$ be a second countable locally compact group. Then there exists a $G$-boundary on which the action is topologically free if and only if no relatively amenable subgroup of $G$ is confined.
\end{prop}

\begin{proof}
The `only if' part is a fairly direct consequence of the definitions; it holds regardless of the second countability assumption. Suppose conversely that the $G$-action on each of the $G$-boundaries is not topologically free. In particular, this is true for the Furstenberg boundary $\partial_F G$. Since $G$ is second countable, the space $\sub(G)$ is metrizable, and it follows from semi-continuity that there exists a point $z \in \partial_F G$ at which the stabilizer map $\partial_F G \to \Sub(G) : x \mapsto G_x$, is continuous \cite[Th.~VII]{Kuratowski1928}. By hypothesis, the stabilizer $G_z$ is relatively amenable and non-trivial. By  \cite[Prop.~1.2]{GW}, the closure of the conjugacy class of $G_z$ in $\Sub(G)$ is a URS. In particular $G_z$ is confined. 
\end{proof}

\begin{rmk}\label{rem:2nd-countable}
The proposition remains true regardless of the second countability assumption. Indeed, in view of the main results of \cite{LB-Tsankov}, for every locally compact group $G$,  the stabilizer map $\partial_F G \to \Sub(G)$ is continuous everywhere. 
\end{rmk}

\section{The proof of Theorem \ref{thm:intro}} \label{sec-loc-bndry}

Let $X$ be a compact $G$-space. Given a subset $\alpha \subseteq X$, we define the \textbf{rigid stabilizer} of $\alpha$ in $G$ as the pointwise fixator of the complement $X \setminus \alpha$. It is denoted by
$$\rist_G(\alpha) = \Fix_G(X \setminus \alpha).$$ 
Let us now assume that $X$ is totally disconnected. We say that the $G$-action on $X$ is \textbf{piecewise minimal} (resp. \textbf{piecewise strongly proximal}) if for every non-empty clopen set $\alpha \subseteq X$, the action of the rigid stabilizer $\rist_G(\alpha)$ on $\alpha$ is minimal (resp. strongly proximal). If the $G$-action has both properties, we say that the $G$-action is \textbf{piecewise minimal-strongly-proximal}; we shall focus on that situation.

The criterion of non-confinement we shall establish in order to prove Theorem \ref{thm:intro} is Theorem \ref{thm:S_X} below. It requires the following definition.

	\begin{defin}
			Let $G$ be a locally compact and $X$ be a  totally disconnected compact $G$-space. We denote by  $\mathcal{S}_X$  the subset of $\sub(G)$ consisting of those closed subgroups $H \leq G$ such that for every non-empty clopen subset $\alpha$ of $X$, the stabilizer of $\alpha$ in $H$ fixes a probability measure on $\alpha$. 
		\end{defin}
		
We observe that  $\mathcal{S}_X$ indeed contains all relatively amenable subgroups of $G$. 
		
		\begin{lem} \label{lem-rel-am-fix-measure-locally}
			Let $G$ be a locally compact and $X$ be a  totally disconnected compact $G$-space. Every relatively amenable closed subgroup of $G$ is contained in $\mathcal{S}_X$.
		\end{lem}
		
		\begin{proof}
			Let $H$ be subgroup of $G$ that is amenable relative to $G$.  If $\alpha$ is a non-empty clopen subset of $X$, then $\stab_G(\alpha)$ is open in $G$, and therefore $\stab_H(\alpha)$ is  amenable relative to $\stab_G(\alpha)$ \cite[Lemma 11]{Cap-Mon-RelMoy}. So  $\stab_H(\alpha)$ fixes a probability measure on $\alpha$, and $H \in \mathcal{S}_X$. 
		\end{proof}

In view Lemma~\ref{lem-rel-am-fix-measure-locally}, we see that Theorem \ref{thm:intro}  is an immediate consequence of the following, which is the main result of this section. Here and in the rest of this paper, we user the abbreviation \textbf{tdlc} for \textit{totally disconnected locally compact}.
		
		\begin{thm}\label{thm:S_X}
			Let $G$ be a tdlc group and $X$ be a totally disconnected compact $G$-space on which the $G$-action is  faithful and piecewise minimal-strongly-proximal. 
Then no subgroup $H \in \mathcal S_X$ is confined.  
		\end{thm}

The proof requires a few preparations. We recall the the \textbf{commensurator} of a subgroup $H$ of a group $G$, denoted by $\Comm_G(H)$, consists of those elements $g \in G$ such that the intersection $H \cap gHg^{-1}$ is of finite index both in $H$ and in $ gHg^{-1}$. We say that 
$H$ is a \textbf{commensurated} subgroup of $G$ if $\Comm_G(H) = G$.

\begin{lem} \label{lem-commens-finite orbits}
	Let $G$ be a group acting on a set $X$ and let $\Lambda$ be a commensurated subgroup of $G$. Then the union of the finite $\Lambda$-orbits in $X$ is a $G$-invariant subset of $X$.
\end{lem}

\begin{proof}
	Let $x \in X$ having a finite $\Lambda$-orbit and let $g \in G$. We want to see that $\Lambda gx$ is finite. Since $g$ commensurates $\Lambda$, there exist $g_1,\ldots,g_r \in G$ such that $\Lambda g$ is covered by $\cup g_i \Lambda$. So $\Lambda gx$ in contained in a finite union of finite sets, and hence is finite.
\end{proof}
		
		\begin{lem} \label{lem-conf-fixedpts}
			Let $G$ be a locally compact group and $X$ be  a totally disconnected locally compact $G$-space $X$  on which the $G$-action is faithful. Assume moreover that  for every non-empty clopen subset $\alpha$ of $X$, the action of $\rist_G(\alpha)$ on $\alpha$ is minimal and proximal. 
For every closed subset $\H \subset \sub(G)$ not containing the trivial subgroup $\{e\}$,  there exists $r \geq 1$ such that every element of $\H$ has at most $r$ finite orbits in $X$. 
		\end{lem}
		
		\begin{proof}
Since $\H$ is Chabauty-closed and does not contain the trivial subgroup $\{e\}$, it follows that there exists a compact subset $P \subset G$ with $e \not \in P$ that intersects non-trivially every element of $\H$. 

Since all elements of $P$ act non-trivially on $X$, by compactness one can find non-empty disjoint compact open subsets $\alpha_0,\ldots,\alpha_{r}$ such that for every $g \in P$ there is $i$ such that $g(\alpha_i)$ and $\alpha_i$ are disjoint. Suppose that there exists $H \in \H$ with $r+1$ distinct fixed points $x_0,\ldots,x_{r}$. Upon replacing $\alpha_0,\ldots,\alpha_{r}$ with smaller subsets, we can assume that $x_i$ does not belong to $\alpha_j$ for all $i,j$. Consider compact open neighbourhoods $\beta_i$ of $x_i$ that are all disjoint and such that $\beta_i$ does not intersect any $\alpha_j$. Since $\rist_G(\alpha_i \cup \beta_i)$ acts minimally on $\alpha_i \cup \beta_i$, there is $g_i \in \rist_G(\alpha_i \cup \beta_i)$ such that $g_i(x_i) \in \alpha_i$. Then $g = g_0 \cdots g_r$ verifies $g(x_i) \in \alpha_i$ for all $i$, so that $g H g^{-1}$ fixes a point inside $\alpha_i$ for all $i$. So $g H g^{-1}$ does not intersect $P$, which is a contradiction.
			
			Suppose now that there exists $H \in \H$ with $r+1$ distinct finite orbits $F_0,\ldots,F_{r}$. Again we choose compact open neighbourhoods $\beta_i$ of $F_i$  that are all disjoint. Using that the action of $\rist_G(\beta_i)$ on $\beta_i$ is proximal, for each $i$, one can find a net $g_{i,k} \in \rist_G(\beta_i)$ and a point $x_i \in \beta_i$ such that $g_{i,k}(x)$ converges to $x_i$ for all $x \in F_i$. Set $g_k = g_{0,k} \cdots g_{r,k}$. Then any accumulation point of $g_k H {g_k}^{-1}$ in $\H$ fixes $x_0,\ldots,x_{r}$, in contradiction with the previous paragraph.
		\end{proof}
	
	Recall that a locally compact group is \textbf{locally elliptic} if every compact subset is contained in a compact subgroup of $G$. The \textbf{locally elliptic radical} $\mathrm{Rad_{LE}}(G)$ of $G$ is the union of all closed normal locally elliptic subgroups of $G$; it is a closed normal subgroup of $G$ \cite{Platonov-ell}. The \textbf{topological FC-center} $B(G)$ of a locally compact group $G$ is the set of elements with a relatively compact conjugacy class. It is a (not necessarily closed) normal subgroup of $G$. More generally, if $H$ is a subgroup of $G$, we let $B_H(G)$ be the set of elements of $G$ with a relatively compact $H$-conjugacy class. Note that when $H$ is normal in $G$, $B_H(G)$ is also normal in $G$. An  elements $g \in G$ is \textbf{periodic} if the subgroup generated by $g$ is relatively compact in $G$.  The set of periodic elements of $G$ is denoted $P(G)$.
	
	The following result was proven by Wu--Yu in \cite{WuYu-B(G)}. 
	
	\begin{prop} \label{prop-B(G)-am}
Let $G$ be a tdlc group such that $B(G)$ is dense in $G$. Then $G / \mathrm{Rad_{LE}}(G)$ is a discrete torsion free abelian group. 

In particular for every tdlc group $G$, the topological FC-center $B(G)$ is contained in the amenable radical of $G$. 
	\end{prop}

\begin{proof}
Assume that $B(G)$ is dense in $G$. Under this assumption, Theorem 4 in \cite{WuYu-B(G)} asserts that $P(G) = \overline{P(G) \cap B(G)}$, that $P(G)$ is a characteristic open subgroup of $G$, and that $G/P(G)$ is torsion free abelian. Hence to conclude it is enough to see that $P(G) = \mathrm{Rad_{LE}}(G)$. By a result of Usakov (see e.g.\ Theorem A in \cite{WuYu-B(G)}), we have $P(G) \cap B(G) \leq \mathrm{Rad_{LE}}(G)$. Since $\mathrm{Rad_{LE}}(G)$ is closed, we infer that $P(G) \leq \mathrm{Rad_{LE}}(G)$. Now $G/P(G)$ has trivial locally elliptic radical, as it is discrete and torsion free abelian, so we actually have $P(G) = \mathrm{Rad_{LE}}(G)$, as desired. 
	
Now for a general tdlc group $G$, the previous statement can be applied to $H = \overline{B(G)}$. This subgroup is therefore an extension of a locally elliptic group by an abelian group, and in particular is amenable. So $B(G)$ lies in the amenable radical of $G$, as desired.
\end{proof}
	
	\begin{prop} \label{prop-H-not-intersect-open}
		Let $G$ be tdlc group, $J$ an open subgroup of $G$ and $N$ a normal subgroup of $G$ contained in $J$. Let $H_0$ be a closed subgroup of $G$ such that $H_0 \cap J = \{e\}$. Let $\H_{0,N}$ be the  closure of the $N$-orbit of $H_0$ in $\sub(G)$, and let $\H' \subseteq \H_{0,N}$ be a closed minimal $N$-invariant subset. Then  $H \leq B_N(G)$ for every $H \in \H'$. 
	\end{prop}

\begin{proof}
We first observe that $H \cap J = \{e\}$ for all $H \in \H_{0,N}$. Indeed, the map $\H_{0,N} \to \sub(J)$, $H \mapsto H \cap J$, is continuous since $J$ is open in $G$, and $N$-equivariant since $N$ is contained in $J$. So the set of subgroup $H$ in $\H_{0,N}$ such that $H \cap J = \{e\}$ is closed and $N$-invariant in $\H_{0,N}$, and hence is equal to $\H_{0,N}$ since it contains $H_0$.

Now let $\H' \subseteq \H_{0,N}$ be a closed minimal $N$-invariant subset. Fix a compact open subgroup $U$ of $G$ contained in $J$, a subgroup $H$ in $\H'$ and an element $h_0$ in $H$. The set 
\[ \mathcal{O} =  \left\lbrace K \in \H' | K \cap h_0U \neq \emptyset \right\rbrace  \] 
is an open subset of $\H'$, which is non-empty since $H \in \mathcal{O} $. So by minimality and compactness one can find $g_1,\ldots,g_r$ in $N$ such that $\H' = \bigcup_i g_i \mathcal{O} g_i^{-1}$. Set $P =  \bigcup_i g_i h_0 U g_i^{-1}$. We infer that $K \cap P \neq \emptyset$ for all $K$ in $\H'$.  

Consider the canonical projection $\pi\colon G\to G/N$. We have $\pi(P) \subset \pi(h_0)\pi(U)$, so that $P \subset h_0 J$ since $UN \leq J$. 
In particular, for any two elements $p_1, p_2 \in P$, we have  $p_1^{-1}p_2 \in J$.  Since $H \cap J$ is trivial, this implies that $H \cap P = \left\lbrace h_0\right\rbrace$. Moreover, since $K \cap J = \{e\}$ for all $K \in \H'$, we see that  $gHg^{-1} \cap h_0 J$ contains at most one element for any $g \in N$. Since $gHg^{-1} \cap P$ is non-empty and since $P \subset h_0 J$, we deduce that $gHg^{-1} \cap h_0 J =  gHg^{-1} \cap P$. On the other hand, we have $gh_0g^{-1} \in gHg^{-1} \cap h_0 J$ since $g \in N \leq J$. Therefore we have $gh_0g^{-1} \in P$. This shows that the $N$-conjugacy class of $h_0$ is entirely contained in the compact set $P$, and hence $h_0 \in B_N(G)$. This is valid for all $h_0 \in H$ and all $H \in  \H'$, so the statement is proved.
\end{proof}

\begin{proof}[Proof of Theorem~\ref{thm:S_X}]
Before going into the proof, let us first observe that if  the space $X$ is finite, then $X$ is a singleton   by strong proximality of the $G$-action on $X$. Since $G$ acts faithfully, it is the trivial group, and the required conclusion is trivially true. We assume henceforth that $X$ is infinite. 
			
 We first observe that if a closed subgroup $H$ of $G$ fixes probability measure $\mu$  on $X$, then by strong proximality of $G$ we can find a net $(g_i)$ in $G$ such that $g_i  \mu$ converges to a Dirac measure $\delta_x$. By upper semi-continuity, every accumulation point of $(H^{g_i})$ in $\sub(G)$ fixes the point $x$. 
 	
Suppose for a contradiction that there exists $H_0 \in \mathcal S_X$ that is confined, and let $\H$ denote the orbit closure of $H_0$ in $\sub(G)$. Hence $\H$ does not contain $\{e\}$. For $K \in \H$, let $n_f(K)$ be the number of 
finite $K$-orbits in $X$. 
According to Lemma~\ref{lem-conf-fixedpts},  the supremum $r = \sup\{ n_f(K) \mid K \in \H\} $ is finite. Let $L$ in $\H$ such that $n_f(L) = r$. Note that $n_f(L) \geq 1$ by the initial observation in the second paragraph of this proof. Let $F \subset X$ be the set of points with a finite $L$-orbit.  We denote by $G_F$ the pointwise fixator of $F$ and by $G_{(F)}$ the setwise stabilizer of $F$ in $G$. We also let $G_F^0$ be the subgroup of $G_F$ consisting of those elements fixing pointwise a neighbourhood of $F$. In other words, the group $G_F^0$ is the union, taken over all clopen subsets $\alpha \subset X$ with $\alpha \cap F = \varnothing$, of the rigid stabilizer $\rist_G(\alpha)$. Note that $G_F$ is a finite index open subgroup of $G_{(F)}$, and that $G_F^0$ is a normal subgroup of $G_{(F)}$. We fix a compact open subgroup $U$ of $G_{(F)}$ contained in $G_F$, and we denote by $J$ the subgroup of $G_{(F)}$ generated by $U$ and $G_F^0$.  Since $G_F^0$ is normal in $G_{(F)}$ and $U$ is commensurated in $G_{(F)}$, the subgroup $J$ is commensurated in $G_{(F)}$. 
 
 Let $\K$ be the subset of $\H$ consisting of those elements $H$ such that $H$ is contained in $G_{(F)}$. Note that $\K$ is closed and $G_{(F)}$-invariant. By the definition of $F$, we have $L \in \K$, so that  $\K$ is non-empty.

			\begin{lem} \label{lem-inter-J-non-trivial}
				We have $H \cap J \neq \{e\}$ for all $H \in \K$. 
			\end{lem}
			
			\begin{proof}[Proof of Lemma \ref{lem-inter-J-non-trivial}]
	Suppose for a contradiction that this is not the case, i.e. there exists $K_0 \in \K$ such that $K_0 \cap J =  \{e\}$.  By invoking Proposition \ref{prop-H-not-intersect-open} for the group $G_{(F)}$, the open subgroup $J$ and the normal subgroup $G_F^0$, we infer that there exists $K \in \K$ such that $K$ is contained in $B_{G_F^0}(G_{(F)})$. 
	
	Now, we claim that the subgroup $B_{G_F^0}(G_{(F)})$ is actually trivial. Otherwise, since $B_{G_F^0}(G_{(F)})$ is a normal subgroup of $G_{(F)}$, the classical commutator lemma for normal subgroups (see e.g.\ \cite[Lem. 4.1]{Nek-expand} for a modern reference) asserts that there exists a compact open subset $\alpha$ of $X \setminus F$ such that the derived subgroup $\rist_G(\alpha)'$ is contained in $B_{G_F^0}(G_{(F)})$. Since $\rist_G(\alpha)$ is contained in $G_F^0$, this actually implies that $\rist_G(\alpha)' $ is contained in  $B(\rist_G(\alpha))$, and hence also in the amenable radical of $\rist_G(\alpha)$ according to Proposition \ref{prop-B(G)-am}. On the other hand, by assumption the action of $\rist_G(\alpha)$ on $\alpha$ is faithful, minimal and strongly proximal, so the group $\rist_G(\alpha)$ has a trivial amenable radical (indeed $\alpha$ is not a singleton since $X$ is infinite by hypothesis). In particular $\rist_G(\alpha)$ has trivial topological FC-center, and $\rist_G(\alpha)' $ is trivial. This implies that $\rist_G(\alpha)$ would be abelian, hence amenable, which again is impossible. It follows that $B_{G_F^0}(G_{(F)})$ is indeed trivial, and hence so is  the subgroup $K$. We deduce that  $\H$  contains the trivial subgroup, which is a contradiction. 
	\end{proof}
			
			We next record the following.
			
			\begin{lem} \label{lem-intersect-finite-orbit}
				For every $H \in \K$, the subgroup $H \cap J$ does not have any finite orbit in $X \setminus F$.
			\end{lem}
			
			\begin{proof}[Proof of Lemma \ref{lem-intersect-finite-orbit}]
				The set \[ \left\lbrace H \cap J | H \in \K \right\rbrace  \] is a closed $J$-invariant subset of $\sub(J)$, which does not contain the trivial subgroup according to Lemma \ref{lem-inter-J-non-trivial}. Let us fix some  $H \in \K$. 
				Since $J$ contains $G_F^0$, we deduce that the $J$-action on  $X \setminus F$ satisfies the hypotheses of Lemma~\ref{lem-conf-fixedpts}. Hence, the number of finite $H \cap J$-orbits is finite. Since $J$ is a commensurated subgroup of $G_{(F)}$, the subgroup $H \cap J$ is commensurated in $H$, and hence the union of the finite orbits of $H \cap J$ is an $H$-invariant set (see Lemma \ref{lem-commens-finite orbits}). By the definition of $\K$, every finite $H$-orbit is contained in $F$. This confirms that every finite  $H \cap J$-orbit is contained in $F$,  as desired.
			\end{proof}
			
			We shall now finish the proof of  Theorem~\ref{thm:S_X}. Let $A$ be the collection of $U$-invariant clopen neighbourhoods of $F$. For $\alpha \in A$, the subgroup $\rist_G(X \! \setminus \! \alpha)$ is normalized by $U$ since $U$ stabilizes $\alpha$, and $J_\alpha := \rist_G(X \! \setminus \! \alpha) U$ is a subgroup of $J$. We have $J_\alpha \subset J_{\alpha'}$ for all $\alpha' \subseteq \alpha$. Since $A$ forms a basis of neighbourhoods of $F$, we have $\bigcup_A \rist_G(X \! \setminus \! \alpha) = G_F^0$, so that $\bigcup_A J_\alpha = J$.
			
			 We fix a point $\xi$ in $X \setminus F$, and $\alpha \in A$ such that $\xi \notin \alpha$. We also fix an element $H \in \K$. The subgroup  $H \cap J_\alpha$ stabilizes the clopen subset $X \! \setminus \! \alpha$, and hence fixes a probability measure $\nu_\alpha$ on $X \! \setminus \! \alpha$ by the assumption that $H$ belongs to $\mathcal{S}_X$. Since the action of $\rist_G(X \! \setminus \! \alpha)$ on $X \! \setminus \! \alpha$ is minimal and strongly proximal, one can find a net $(g_k)$ in $\rist_G(X \! \setminus \! \alpha)$  such that $g_k(\nu_\alpha)$ converges to the Dirac measure at $\xi$. By compactness we may assume that $H^{g_k}$ converges to a point $H^{(\alpha)}$ in $\K$. 
			 Since $J_\alpha$ is open, it follows that  $H^{g_k} \cap J_\alpha$ converges to  $H^{(\alpha)} \cap J_\alpha$, which must therefore fix $\xi$ by the choice of the net $(g_k)$. 
			 
			 Now by compactness again, upon passing to a subnet we may assume that $H^{(\alpha)}$ converges in $\K$. Let $K$ be its limit. We claim that $K \cap J$ fixes $\xi$. Since $J = \bigcup_A J_\beta$, it suffices to check that $K \cap J_\beta$ fixes $\xi$ for all $\beta \in A$. We fix $\beta \in A$. Using again that $J_\beta$ is open, the subgroup  $K \cap J_\beta$ is the limit of $H^{(\alpha)} \cap J_\beta$. Now eventually  we have $J_\beta \leq J_{\alpha}$. Therefore $H^{(\alpha)} \cap J_\beta$ is a subgroup of $H^{(\alpha)} \cap J_\alpha$, and hence $H^{(\alpha)} \cap J_\beta$ fixes $\xi$. By upper semi-continuity, so does $K \cap J_\beta$. This proves the claim. So $K$ is an element of $\K$ such that $K \cap J$ fixes $\xi$. Since Lemma~\ref{lem-intersect-finite-orbit} above says that no element of $\K$ can have this property, we have obtained a contradiction, thereby finishing  the proof of the theorem.
		\end{proof}

As mentioned in the introduction, the conclusion of Theorem~\ref{thm:S_X} was already known when $G$ is a discrete group, as it follows from Theorem 1.1 of \cite{LBMB-comm-lemma}. This result asserts that when $G$ is discrete and $X$ is a $G$-space on which $G$ acts faithfully, then every confined subgroup of $G$ contains the commutator subgroup of the rigid stabilizer of some non-empty open subset of $X$. Here we note that this much stronger conclusion no longer holds outside the realm of discrete groups. For instance let $G = \mathrm{Aut}(T_d)$ be the automorphism group of a $d$-regular tree, $d \geq 3$. The group $G$ acts faithfully on $\partial T$. If $\Gamma$ is a discrete and cocompact subgroups for $G$ (take for instance $\Gamma$ to be the free product of $d$ copies of the cyclic group $C_2$), then $\Gamma$ has a closed conjugacy class in $\sub(G)$ because $\Gamma$ is cocompact. So in particular $\Gamma$ is a confined subgroup of $G$. But on the other hand the pointwise fixator in $\Gamma$ of an infinite subtree is trivial; and hence $\Gamma$ obviously cannot contain the commutator subgroup the rigid stabilizer of a non-empty open subset of $\partial T$.

\section{Neretin groups} \label{sec-neretin}

\subsection{Definitions} \label{subsec-ner-def}

We briefly review basic notions concerning Neretin groups (see  \cite[\S6.3]{CapDM} and \cite{GarnLaza} for more details). Let $d, k \geq 2$ be integers, and let $\mathcal T_{d, k}$ be a rooted tree such that the root has $k$ descendants, and every vertex distinct from the root has $d$ descendants. Notice that $\mathcal T_{d, d}$ is the regular rooted tree of degree $d$. Moreover, the graph $\mathcal T_{d, d+1}$ is isomorphic to the regular non-rooted tree of degree $d+1$.  An \textbf{almost automorphism} of $\mathcal T_{d, k}$ is a triple of the form $(A, B, \varphi)$, where $A$ and $B$ are finite subtrees of $\mathcal T_{d, k}$ containing the root such that $| \partial  A | = | \partial  B |$, and $\varphi$ is an isomorphism of forests $\mathcal T_{d, k} \setminus A \to \mathcal T_{d, k} \setminus B$. The \textbf{group of almost automorphisms} of $\mathcal T_{d, k}$, denoted by 
$\mathcal N_{d, k}$, is the quotient of the set of all almost automorphisms by the relation which identifies two almost automorphisms $(A, B, \varphi)$ and $(A', B', \varphi')$ if there exists some finite subtree $A''$ containing $A \cup A'$ and such that $\varphi$ and $\varphi'$ coincide on $\mathcal T_{d, k} \setminus A''$.   It is easy to verify that  $\mathcal N_{d, k}$ is indeed a group. It carries a unique locally compact group topology such that the natural inclusion of the profinite group $\Aut(\mathcal T_{d, k})$ in $\mathcal N_{d, k}$ is continuous and open. In particular $\mathcal N_{d, k}$  is a totally disconnected locally compact group. Moreover, as observed in \cite[\S6.3]{CapDM}, the group $\mathcal N_{d, k}$  has a finitely generated dense subgroup, and is thus compactly generated.

Observe that for all integers $d, k$, the topological group $\mathcal N_{d, k}$ is naturally isomorphic to $\mathcal N_{d, k + d - 1}$, so that  the values of $k$ in the set $\{2, 3, \dots, d\}$ suffice to account for all isomorphism classes of the Neretin groups.

The collection of those elements of $\mathcal N_{d, k}$ that can be represented by an almost automorphism $(A, B, \varphi)$ such that $A = B$ is a ball around the root of $\mathcal T_{d, k}$ is a subgroup of $\mathcal N_{d, k}$ that we denote by $\mathcal O_{d, k}$. Clearly, the full automorphism group $\Aut(\mathcal T_{d, k})$ of the rooted tree is a compact open subgroup of $\mathcal N_{d, k}$  contained in $\mathcal O_{d, k}$. The group $\mathcal O_{d, k}$ is the directed union of its compact open subgroups and is thus amenable. Moreover $\mathcal O_{d, k}$ is  topologically simple (see \cite[Rem.~6.8 and Lem.~6.9]{CapDM}). 

The set of ends $\partial \mathcal T_{d, k}$ is a compact $\mathcal N_{d, k}$-space on which  $\mathcal N_{d, k}$ acts faithfully. Given a vertex $v$ of $\mathcal T_{d, k}$, we denote by $\mathcal T_v$ the subtree spanned by all vertices separated from the root by $v$. In particular, if $v_0$ denotes the root, we have $\mathcal T_{v_0} = \mathcal T_{d, k}$. For each vertex $v$, the set $\partial \mathcal T_v$ is a basic clopen subset of $\partial \mathcal T_{d, k}$; the collection of those clopen subsets forms a basis of the topology on $\partial \mathcal T_{d, k}$. 

\subsection{Application of Theorem \ref{thm:intro}}

The following important feature follows from the self-replicating properties of the Neretin groups.

\begin{lem}\label{lem:StrongProx}
The action of $\mathcal N_{d, k}$ on $\partial \mathcal T_{d, k}$ is piecewise minimal-strongly-proximal. 
\end{lem}
\begin{proof}
Set $G = \mathcal N_{d, k}$. 
Let us first show that, for every $d$ and $k$, the $G$-action on $\partial \mathcal T_{d, k}$ is minimal and strongly proximal. 

Clearly the action is transitive, hence minimal. It is easy to see that the action is $n$-transitive (i.e. transitive on ordered $n$-tuples of distinct points) for all $n$. In particular it is proximal. To show strong proximality, it suffices by \cite[Prop.~VI.1.6]{Margulis} to show the existence of a compressible non-empty open set. The existence of such an open set is clear since $G$ acts transitively on the basic clopen sets  arising as the set of ends of the sub-rooted trees of the form $\mathcal T_v$, where $v$ is a vertex different from the root. 

Let now $\alpha$ be a non-empty clopen subset of $\partial  \mathcal T_{d, k}$. Then there is a finite set of vertices $\{v_1, \dots, v_n\}$ distinct from the root, such that $\alpha$ decomposes as the disjoint union $\bigcup_{i=1}^n \partial \mathcal T_{v_i}$. Moreover, it follows from the definitions that $\rist_G(\alpha)$ is isormorphic to $\mathcal N_{d, n}$, and that  $\alpha$ is equivariantly homeomorphic to $\partial T_{d, n}$. Therefore, it follows from the first part of the proof that the $\rist_G(\alpha)$-action on $\alpha$ is minimal and strongly proximal. 
\end{proof}

As announced in the introduction, the Neretin groups constitute an important family of examples to which Theorem~\ref{thm:intro} applies.

\begin{proof}[Proof of Corollary~\ref{cor:intro:Neretin-confined}]
The $\mathcal N_{d, k}$-action on the compact space $X = \partial \mathcal T_{d, k}$ is continuous and faithful. Moreover it is piecewise minimal-strongly proximal by Lemma~\ref{lem:StrongProx}. Thus the conclusion follows from Theorem~\ref{thm:intro}.
\end{proof}

We underline that for $X = \partial \mathcal T_{d, k}$, the Neretin group  $\mathcal N_{d, k}$  has a discrete subgroup $F_{d, k}$, also called  Thompson's group $F_{d, k}$,  which belongs to $\mathcal S_X$, while it is an open problem to determine whether $F_{d, k}$ is relatively amenable in $\mathcal N_{d, k}$ (not to mention the problem whether or not $F_{d, k}$ is amenable). We refer to \cite[\S6.3]{CapDM} for the definition.  

\begin{cor}\label{cor:Thompson_F}
For all integers $d, k \geq 2$, Thompson's group $F_{d, k}$ is not confined in the Neretin group $\mathcal N_{d, k}$. 
\end{cor}
\begin{proof}
Let $X = \partial \mathcal T_{d, k}$ and $\Gamma = F_{d, k}$. It follows from the definition of $F_{d, k}$ that for any non-empty clopen subset $\alpha \subseteq X$, the stabilizer $\Stab_\Gamma(\alpha)$ fixes a point in $\alpha$. In particular $\Stab_\Gamma(\alpha)$ fixes a Dirac probability measure on $\alpha$. Thus $\Gamma \in \mathcal S_X$, and the conclusion follows from Theorem~\ref{thm:S_X}.
\end{proof}

\begin{rmk}\label{rem:Gelfand}
N.~Monod \cite{Monod} has proved that a locally compact group $G$ having a compact subgroup $K$ such that  $(G, K)$ forms a Gelfand pair, must have a cocompact amenable closed subgroup. Y.~Neretin \cite[Th.~1.2]{Neretin} has proved that the group $\mathcal N_{d, 2} \cong \mathcal N_{d, d+1}$ has an open subgroup $A$ such that $(\mathcal N_{d, 2}, A)$ forms a generalized Gelfand pair. In view of Corollary~\ref{cor:intro:Neretin-confined}, we see that Monod's result cannot be extended from Gelfand pairs to generalized Gelfand pairs, even among simple groups. Without the condition of simplicity, this can be observed by a much more straightforward argument. Indeed, given a locally compact group $G$ and a subgroup $A$ containing a closed normal subgroup $N$ of $G$, every irreducible unitary representation of $G$ with a  nonzero $A$-invariant vector factorizes through a representation of $G/N$. In particular $(G, A)$ is a generalized Gelfand pair if and only if $(G/N, A/N)$ is one. Taking $N$ to a discrete free group of countable rank, $G = N \rtimes \mathrm{SL}_2(\mathbf Q_p)$ (where  $\mathrm{SL}_2(\mathbf Q_p)$ acts by permuting continuously the elements of a free basis of $N$) and $A = N \rtimes \mathrm{SL}_2(\mathbf Z_p)$, we deduce that $(G, A)$ is a generalized Gelfand pair (because $(\mathrm{SL}_2(\mathbf Q_p), \mathrm{SL}_2(\mathbf Z_p))$ is a Gelfand pair) such that $G$ does not have any cocompact amenable subgroup. 
\end{rmk}

\subsection{Quasi-regular representations}

We refer to \cite[Appendix F]{BHV} for general background about the Fell topology and weak containment of unitary representations. 

\begin{prop}\label{prop:Fell}
Let $G$ be a second countable locally compact group and $H \leq G$ be a closed subgroup. If the closure of the conjugacy class of $H$ in $\sub(G)$ contains the trivial subgroup $\{e\}$, then the quasi-regular representation $\lambda_{G/H}$ weakly contains the regular representation $\lambda_G$. 
\end{prop}

\begin{proof}
By a result of Fell \cite{Fell1964}, if a sequence $(H_n)$ of closed subgroups converges to $J$ in $\sub(G)$, the corresponding sequence of quasi-regular representations $(\lambda_{G/H_n})$ converges to $\lambda_{G/J}$ in the Fell topology. By hypothesis, there exists a sequence $(H_n)$ of conjugates of $H$ in $G$ such that $\lim_n H_n = \{e\}$. Since $H_n$ is conjugate to $H$, the quasi-regular representation $\lambda_{G/H_n}$ is equivalent to $\lambda_{G/H}$. Hence, we deduce from Fell's result that the constant sequence $(\lambda_{G/H})$ converges to the regular representation $\lambda_G$ in the Fell topology. This is a reformulation of the fact that  $\lambda_{G/H}$ weakly contains $\lambda_G$. 
\end{proof}

Combining this with Corollary~\ref{cor:intro:Neretin-confined}, we obtain the following.

\begin{cor}\label{cor:weaklyregular}
Let $d, k \geq 2$ be integers. For any  relatively amenable subgroup  $A$ of $G = \mathcal N_{d, k}$,   the quasi-regular representation $\lambda_{G/A}$ is weakly equivalent to the regular representation $\lambda_G$. 
\end{cor}

\begin{proof}
Since $A$ is relatively amenable, the representation $\lambda_{G/A}$ is weakly contained in $\lambda_G$. By Corollary~\ref{cor:intro:Neretin-confined}, the  closure of the conjugacy class of $A$ in $\sub(G)$ contains the trivial subgroup $\{e\}$. By Proposition~\ref{prop:Fell}, this implies that $\lambda_G$ is weakly contained in $\lambda_{G/A}$. 
\end{proof}

The following proposition records   a special case of well known results due to G.~Mackey~\cite{Mackey1951}. 

\begin{prop}\label{prop:Mackey}
Let $G$ be a second countable locally compact group and $H \leq G$ be an open subgroup. 

\begin{enumerate}[label=(\roman*)]
\item If $\Comm_G(H) = H$, then the unitary representation $\lambda_{G/H}$ is irreducible. 

\item If $J \leq G$ is an open subgroup such that every $H$-orbit on $G/J$ is infinite, then  the unitary representations $\lambda_{G/H}$ and $\lambda_{G/J}$ are not equivalent.
\end{enumerate}
	
\end{prop}
\begin{proof}
The first assertion follows from \cite[Theorem~$6'$]{Mackey1951}. To prove the second assertion, it suffices to observe that $H$ has a non-zero invariant vector in $\ell^2(G/H)$, while the hypothesis implies that the only $H$-invariant vector in $\ell^2(G/J)$ is the zero function. 
\end{proof}
		
We can now prove the following. 

\begin{prop}\label{prop:Ineq}
For all integers $d, k \geq 2$, the Neretin group $G = \mathcal N_{d, k}$ has two open amenable subgroups $O, O'$ such that the unitary representations $\lambda_{G/O}$ and $\lambda_{G/O'}$ are irreducible and inequivalent.
\end{prop}
\begin{proof}
Let $K$ be the full automorphism group of the rooted tree $\mathcal T_{d, k}$. Thus $K$ is a compact open subgroup of $\mathcal N_{d, k}$ acting transitively on the set of ends $X = \partial \mathcal T_{d, k}$.   Thus $K$ fixes a unique probability measure $\nu$ on $X$. 

Let $O$ denote the stabilizer of $\nu$ in $G$, namely $O = G_\nu$. Given  vertices  $v, w$ of $\mathcal T_{d, k}$, the set of ends $\partial \mathcal T_v$ and $\partial \mathcal T_w$ are basic clopen subsets of $X$. Moreover, we have $\nu(\partial \mathcal T_v) = \nu(\partial \mathcal T_w)$ if and only if $v$ and $w$ have the same level (i.e. their distance from the root are equal). This implies that  any element of $O $ can be represented by an almost automorphism $(A, B, \varphi)$ such that $A = B$ is a ball around the root. It follows that $O = \mathcal O_{d, k}$ (\S  \ref{subsec-ner-def}). Hence $O$ is the directed union of its compact open subgroups; in particular $O$ is amenable. Using \cite[Rem.~6.8 and Lem.~6.9]{CapDM}, we see that $O$ is topologically simple.  In particular, it does not have any proper open subgroup of finite index. It follows that $\Comm_G(O) = N_G(O)$. Since $\nu$ is the unique $K$-invariant probability measure on $X$, it is also the unique $O$-invariant probability measure. This implies that $N_G(O) \leq G_{\nu} = O$. Thus $\lambda_{G/O}$ is irreducible by Proposition~\ref{prop:Mackey}(i). 

Let now $v_1, \dots, v_k$ denote the $k$ descendants of the root of $\mathcal T_{d, k}$. Set $\alpha_i = \partial \mathcal T_{v_i}$. Set $O_i = \rist_O(\alpha_i)$.  The group 
$$P = \langle O_1 \cup \dots \cup O_k \rangle$$ 
is    isomorphic to the direct product $ O_1 \times \dots \times O_k$. We set $O' = N_O(P)$.  

Notice that $O'$ contains $K$, hence it is open. Moreover, we have $O' \leq O$ so that $O'$ is amenable. 
For each $i$, the sub-tree $\mathcal T_{v_i}$ is  naturally isomorphic to $\mathcal T_{d, d}$ and that map defines an isomorphism of $O_i$ to the group $\mathcal O_{d, d}$ defined above. Therefore, the group $O_i$ is topologically simple, so that $P$ does not have any proper open subgroup of finite index. It follows that  $\Comm_G(O')$ normalizes $P$, and thus acts by permutation on the $k$ direct factors $O_1, \dots, O_k$ by the Krull--Remak--Schmidt theorem (see \cite[3.3.8 and 3.3.11]{Robinson}). Since $O_i$ fixes a unique probability measure $\nu_i$ supported on $\alpha_i$, we deduce that $\Comm_G(O')$ permutes the measures $\nu_i$, hence it fixes $ \nu$, since we have $\nu = \frac 1 k \sum_{i=1}^k \nu_i$. Therefore we have $\Comm_G(O') \leq G_\nu =O$. This finally shows that $\Comm_G(O') = N_O(P) = O'$,  so that $\lambda_{G/O'}$ is irreducible by Proposition~\ref{prop:Mackey}(i).

In order to show that $\lambda_{G/O}$  and $\lambda_{G/O'}$ are not equivalent, it suffices to show by Proposition~\ref{prop:Mackey}(ii) that every $O$-orbit on $G/O'$ is infinite. Since $O$ is topologically simple, the existence of a finite $O$-orbit on $G/O'$ implies the existence of a fixed point, which implies in turn that $O $ is contained in some conjugate $gO'g^{-1}$. Since $O$ and $O'$ are both amenable and since $\nu$ is the unique $O$-invariant measure on $\partial T$, this implies that $O = gO'g^{-1}$. This is impossible since $O$ is topologically simple but $O'$ is not.
\end{proof}
	
We can now complete the proof of Theorem~\ref{thm:nonType1}. 

\begin{proof}[Proof of Theorem~\ref{thm:nonType1}]
By Proposition~\ref{prop:Ineq}, the group $G = \mathcal N_{d, k}$ has two inequivalent irreducible unitary representations of the form $\lambda_{G/O}$ and $\lambda_{G/O'}$, with $O, O'$ open amenable subgroups of $G$. In view of Corollary~\ref{cor:weaklyregular}, those representations are both weakly equivalent to the regular representation. By Glimm's theorem (see \cite[Theorem~1]{Glimm1961}), in a second countable locally compact group of type~I, any two weakly equivalent irreducible unitary representations are equivalent. It follows that $G$ is not a type~I group.
\end{proof}
	
\begin{rmk}
We believe that the amenable group $\mathcal O_{d, k}$, which is an open subgroup of $\mathcal N_{d, k}$, is not a type~I group either. This conjectural statement is formally stronger than Theorem~\ref{thm:nonType1} since the type~I property is inherited by open subgroups\footnote{This conjecture has been confirmed by R.~Arimoto~\cite{Arimoto} after the first version of the present paper was circulated.}. That conjecture, and the more ambitious problem of classifying the irreducible unitary CCR representations of $\mathcal O_{d, k}$,  raises interesting questions on the asymptotic representation theory of the finite symmetric groups. 
\end{rmk}

\subsection{Cocompact subgroups and Chabauty-isolation}\label{sec:coco}

In view of the result from \cite{BCGM} and its far-reaching generalization obtained in \cite{Zheng}, as well as Corollary \ref{cor:intro:Neretin-confined}, it is a natural problem to try to classify the confined subgroups of the group $\mathcal N_{d, k}$. Since closed cocompact subgroups are in a sense the most basic examples of confined subgroups, classifying these is a natural first step towards this problem.  The aim of this section, which is independent of the previous ones, is to establish the following:

\begin{thm}\label{thm:cocoNer}
Let $d, k \geq 2$ be integers. Any proper closed cocompact subgroup $H$ of $G= \mathcal N_{d, k}$ fixes a point in $\xi \in \partial  \mathcal T_{d, k}$. Moreover $H$ is a finite index subgroup of $G_\xi$.

In particular, the only closed cocompact unimodular subgroup of $G$ is $G$ itself. 
\end{thm}

It should be noted that the stabilizer $G_\xi$ of a point $\xi \in \partial \mathcal T_{d, k}$ does indeed have proper open subgroups of finite index. Indeed $G_\xi$ admits an infinite cyclic quotient. This can be seen for example by observing that $G_\xi$ is compactly generated (because it is cocompact in the compactly generated group $G$) but not unimodular (because $\partial \mathcal T_{d, k}$ does not carry any $G$-invariant probability measure by Lemma~\ref{lem:StrongProx}).

We also record the following consequence of independent interest. 
\begin{cor}\label{cor:Neretin=Chab-isol}
For all integers $d, k \geq 2$, the Neretin group $\mathcal N_{d, k}$ is an isolated point of the Chabauty space $\sub(\mathcal N_{d, k})$. 
\end{cor}

\begin{proof}
A sequence $(H_n)$ of proper closed subgroups converging to $G = \mathcal N_{d, k}$ is eventually cocompact by \cite[VIII.5.3, Proposition 6]{Bourbaki_Int7-8} since $G$ is compactly generated. Hence upon extracting we may assume that $H_n$ fixes a point $\xi_n \in \partial \mathcal T_{d, k}$ for all $n$ by Theorem~\ref{thm:cocoNer}. The limit $G = \lim_n H_n$ must then fix any accumulation point of $(\xi_n)$ in $\partial \mathcal T_{d, k}$, which is absurd since $G$ is transitive on $\partial \mathcal T_{d, k}$.
\end{proof}

\begin{rmk} \label{rmk-aut-non-isolated}
Corollary~\ref{cor:Neretin=Chab-isol} lies in sharp contrast with the case of the full automorphism group $\Aut(T)$ of a non-rooted regular tree of degree $n \geq 4$, which is not Chabauty-isolated. Indeed the group $\Aut(T)$ can be approximated by a sequence of closed boundary-$2$-transitive subgroups of $\Aut(T)$  (see the Appendix to \cite{CapRad}). 
\end{rmk}

The proof of Theorem~\ref{thm:cocoNer} follows a similar strategy as in \cite{BCGM}. The key step is provided by the following. 

\begin{prop}\label{prop:cocompact-in-O}
Let $d, k$ be integers. The only  proper cocompact closed subgroups of $O = \mathcal O_{d, k}$ are the stabilisers $O_\xi$ of points $\xi \in \partial \mathcal T_{d, k}$.
\end{prop}

The proof requires subsidiary facts on finite permutation groups.

\subsubsection{Factorizations of finite symmetric groups}

The following result is based on a refinement of the arguments from \S3 of \cite{BCGM}. 

\begin{prop}\label{prop:Factor-Sym-group}
Let $\nn = \{1, \dots, n\}$. Assume that $n$ is large enough so that the interval $[n/2, n]$ contains at least~$3$ primes (we remark that the Prime Number Theorem implies that the number of primes between $n/2$ and $n$ tends to infinity with $n$). 

Let  $A, B $ be subgroups of the symmetric group $ \Sym(\nn)$ such that 
$$\Sym(\nn) = AB,$$
where $AB= \{ab \mid a \in A, b\in B\}$.
 If $A$ is transitive and imprimitive on $\nn$, then either $B$ fixes a point  $x \in \nn$ and contains the full alternating group $\Alt(\nn \setminus \{x\})$, or $B$ contains $\Alt(\nn)$. 
\end{prop}

We first collect a couple of preliminaries. 
The following one is a theorem of C.~Jordan \cite{Jordan}.

\begin{prop}\label{prop:Jordan}
Let $\kk = \{1, \dots, k\}$. 
A primitive (hence transitive) subgroup of $\Sym(\kk)$ containing a prime cycle of order $p \leq k-3$ contains the full alternating group $\Alt(\kk)$.
\end{prop}
\begin{proof}
See Theorem~13.9 in \cite{Wielandt}. 
\end{proof}

The following result appears implicitly in \S3 of \cite{BCGM}. 

\begin{lem}\label{lem:Factor-Sym-group}
Let $\nn = \{1, \dots, n\}$ and $A, B $ be subgroups of the symmetric group $ \Sym(\nn)$ such that $\Sym(\nn) = AB$. We assume that $A$ is transitive and imprimitive on $\nn$. If the interval $[n/2, n]$ contains at least~$3$ primes, then there exists a subset $\Omega \subset \nn$ of size $k \geq n/2 + 4$ such that $B$ contains the full alternating group $\Alt(\Omega)$. 
\end{lem}

For the sake of completeness, we reproduce the proof borrowed from \cite{BCGM}. 
\begin{proof}[Proof of Lemma~\ref{lem:Factor-Sym-group}]
Since $A$ is transitive but imprimitive, it is contained in a subgroup of $\Sym(\nn)$ of the form $\Sym(d) \wr \Sym(n/d)$ for some proper divisor $d$ of $n$. In particular no prime in the interval $[\frac{n+1}{2},n]$ divides the order of $A$. Therefore, the transitivity of $B$ on $\Sym(\nn)/A$ implies that any such prime divides $|B|$. The hypothesis made on $n$ implies that $[\frac{n+1}{2},n]$ contains two distinct primes, say $p<q$. Let $\alpha, \beta \in B$ be two elements of order $p$ and $q$. Then $\alpha$ and $\beta$ act as $p$- and $q$-cycles on $\nn$, whose supports intersect non-trivially. Let $\Omega$ be the union of their support. We have $|\Omega|  \geq q+1 \geq p+3 \geq n/2+4$. The subgroup of $B$ generated by $\alpha$ and $\beta$ is transitive on $\Omega$. Therefore  it is also primitive, since it contains a $p$-cycle with $p> |\Omega|/2$. We infer from  Proposition~\ref{prop:Jordan}  that it contains $\Alt(\Omega)$.
\end{proof}

\begin{proof}[Proof of Proposition~\ref{prop:Factor-Sym-group}]
Throughout, we set $G = \Sym(\nn)$.

We first claim that $B$ has at most two orbits on $\nn$, one of which is of size $\geq n-1$. If this were not the case, then $B$ would stabilise a proper subset of $\nn$ for size $m$, with $2 \leq m \leq n/2$. It then follows that $B$ is contained in a subgroup $M \leq G$ of the form $M \cong \Sym(m) \times \Sym(n-m)$. Since $A$ is transitive on $G/B$ by hypothesis, it follows that $A$ is also transitive on $G/M$, which is naturally in one-to-one correspondence with the set of subsets of size $m$   in $\nn$. By hypothesis $A$ is imprimitive on $\nn$, and it is easy to see that it can therefore not act transitively on the set of unordered $m$-tuples in $\nn$ (recalling that $2 \leq m \leq n/2$).  This proves the claim. 

Let now $\mathbf o \subseteq \nn$ be the largest $B$-orbit. By the claim we have either $\mathbf o = \nn$ or $\mathbf o = \nn \setminus \{x\}$ for some $B$-fixed point $x \in \nn$.
We then invoke Lemma~\ref{lem:Factor-Sym-group}, which has two consequences: the $B$-action on  $\mathbf o$  is primitive (since any invariant $B$-invariant partition on $\mathbf o$ induces a $\Stab_B(\Omega)$ invariant partition on $\Omega$), and it contains a $3$-cycle. Therefore Proposition~\ref{prop:Jordan} implies that $ B$ contains  $\Alt(\mathbf o)$. This finishes the proof. 
%
%
%
%
\end{proof}

\subsubsection{Cocompact subgroups of $\mathcal O_{d, k}$ and $\mathcal N_{d, k}$}

\begin{proof}[Proof of Proposition~\ref{prop:cocompact-in-O}]
Set $O = \mathcal O_{d, k}$ 
 and let $H < O$ be a  cocompact closed subgroup. We view $O$ as an ascending chain of the compact open subgroups $O_n$ defined as in \cite{BCGM}, namely $O_n$ consists of those elements represented by almost automorphisms of the form $(A, A, \varphi)$ where $A$ is the $n$-ball around the root of $\mathcal T_{d, k}$. By compactness, there exists $n_0$ such that $O = O_n H$ for all $n \geq n_0$. 

Let $H_n = O_n \cap H$. We have $O_n = O_{n_0} H_n$ for all $n \geq 0$. Let now $U_n \leq O_n$ be the closed normal subgroup consisting of those elements represented by almost automorphisms of the form $(A, A, \varphi)$ where $A$ is the $n$-ball around the root of $\mathcal T_{d, k}$ and $\varphi$ stabilizes each connected component of the forest $\mathcal T_{d, k} \setminus A$.  The quotient $O_n /U_n$ is a symmetric group $\Sym(k_n)$ whose degree $k_n$ tends to infinity with $n$. Let $A_n$ (resp. $B_n$) denote the image of $O_{n_0}$ (resp. $H_n$) in the quotient $O_n/U_n$, which we identity with $\Sym(k_n)$. We have $\Sym(k_n) = A_n B_n$. Moreover $A_n$ is transitive but not primitive for any $n>n_0$. Proposition~\ref{prop:Factor-Sym-group} implies that $B_n$ contains $\Alt(k_n)$ or fixes a point and contains $\Alt(k_n-1)$. 

The first case implies that the image of $H\cap O_{n-1}$ is the full quotient $O_{n-1}/U_{n-1} \cong \Sym(k_{n-1})$. If that happens for infinitely many values of $n$, then we have $(H\cap O_{n-1})U_{n-1} = O_{n-1}$ for infinitely many $n$'s, so that $H U_{n-1} = O$ for infinitely many $n$'s. This means that $H$ is dense in $O$, hence $H=O$ since $H$ is closed. 

Otherwise, we deduce that $H$ fixes a point $\xi \in
\partial \mathcal T_{d, k}$ and  a similar argument shows that $H$ contains  a dense subgroup of the stabilizer $O_\xi$. Thus $H = O_\xi$. 
\end{proof}

\begin{proof}[Proof of Theorem~\ref{thm:cocoNer}]
Set $G = \mathcal N_{d, k}$ and let $H \leq G $ be a   cocompact closed subgroup. Since $O = \mathcal O_{d, k}$ is open in $G$, it has open orbits in $G/H$, hence finitely many orbits, hence clopen orbits. In particular $OH/H \cong O /O \cap H$ is compact. 

By Proposition~\ref{prop:cocompact-in-O}, we have $O\cap H = O$ or $O\cap H = O_\xi$ for some $\xi \in \partial \mathcal T_{d, k}$. The former case implies that $H$ is open, hence of finite index, in $G$. Since $G$ is simple (see \cite{Kapoudjian} for the case where $k = d+1$; in the general case, a similar argument can be applied, exploiting that $O$ is topologically simple by \cite[Rem.~6.8 and Lem.~6.9]{CapDM}), this implies $H=G$, and we are done in this case. 

We assume henceforth that $O \cap H = O_\xi$. 

We next claim $H$ fixes $\xi$. Indeed,  there would otherwise exist  $h \in H$ which does not fix $\xi$. We then see that $H$ contains $\langle O_\xi \cup hO_\xi h^{-1}\rangle$. Since $O$ is $2$-transitive (in fact $\infty$-transitive) on $\partial \mathcal T_{d, k}$, we infer that $H$ is transitive on $\partial \mathcal T_{d, k}$. Therefore $G = G_\xi H$. Hence $G_\xi /G_\xi \cap H$ is homeomorphic to $G/H$, and is thus compact. Since $G_\xi \cap H$ contain $O_\xi = O \cap G_\xi$, it is relatively open in $G_\xi$. Therefore $G/H \cong G_\xi/G_\xi \cap H$ is finite. As before, this implies that $H=G$,  contradicting that $O \cap H = O_\xi$.  

This proves that $H$ fixes $\xi$. We have $O_\xi \leq H \leq G_\xi$, so that $H$ is relatively open in $G_\xi$. Since $H$ is cocompact in $G$, it is cocompact in $G_\xi$, hence of finite index. The result follows. 

For the last claim, it remains to observe that the stabiliser $G_\xi$ is not unimodular, because $G$ does not preserve any probability measure on $\partial \mathcal T_{d, k}$ by Lemma~\ref{lem:StrongProx}. Therefore, if $H$ is unimodular, we must have $H = G$.
\end{proof}

{\small 
\bibliographystyle{abbrv}
\bibliography{Ner_rep}
}
\end{document}